\documentclass[12pt]{amsart}
\usepackage{ amsmath, amsthm, amsfonts, amssymb, color, xcolor}
 \usepackage{mathrsfs}
\usepackage{amsfonts, amsmath}
 \usepackage{amsmath,amstext,amsthm,amssymb,amsxtra,verbatim}
 \allowdisplaybreaks
\begin{document}
\baselineskip 18pt
\hfuzz=6pt

\newtheorem{theorem}{Theorem}[section]
\newtheorem{prop}[theorem]{Proposition}
\newtheorem{lemma}[theorem]{Lemma}
\newtheorem{definition}[theorem]{Definition}
\newtheorem{cor}[theorem]{Corollary}
\newtheorem{example}[theorem]{Example}
\newtheorem{remark}[theorem]{Remark}
\newcommand{\ra}{\rightarrow}
\renewcommand{\theequation}
{\thesection.\arabic{equation}}
\newcommand{\ccc}{{\mathcal C}}
\newcommand{\one}{1\hspace{-4.5pt}1}

\newcommand{\wh}{\widehat}
\newcommand{\f}{\frac}
\newcommand{\df}{\dfrac}
\newcommand{\sgn}{\textup{sgn\,}}
 \newcommand{\rn}{\mathbb R^n}
  \newcommand{\si}{\sigma}
  \newcommand{\ga}{\gamma}
   \newcommand{\nf}{\infty}
\newcommand{\p}{\partial}
\newcommand{\De}{\Delta}

\newcommand{\norm}[1]{\left\|{#1}\right\|}%
\newcommand{\supp}{\operatorname{supp}}

\newcommand{\tf}{\tfrac}
\newcommand{\qq}{\quad\quad}
\newcommand{\lab}{\label}
\newcommand{\zzz}{\mathbf Z}
\newcommand{\li}{L^{\infty}}
\newcommand{\intrn}{\int_{\rn}}
\newcommand{\qqq}{\quad\quad\quad}


\newcommand{\vp}{\varphi}
\newcommand{\al}{\alpha}
\newcommand{\R}{\RR}
\newcommand{\intr}{\int_{\R}}
\newcommand{\intrr}{\int_{\R^2}}
\newcommand{\de}{\delta}
\newcommand{\om}{\omega}
\newcommand{\Tht}{\Theta}
\newcommand{\tht}{\theta}
\newcounter{question}
\newcommand{\qt}{%
        \stepcounter{question}%
        \thequestion}
\newcommand{\bq}{\fbox{Q\qt}\ }
\renewcommand{\wp}{\Psi}
\newcommand{\sgg}{\si_{gg}}
\newcommand{\sbtj}{\si_{b2j}}
\newcommand{\sbrk}{\si_{b3k}}
\newcommand{\sbzg}{\si_{b_0g}}

\newcommand{\abs}[1]{\left\vert #1\right\vert}%

\newcommand{\be}{\beta}%

\def\RR{\mathbb R}
\def\bbr{\mathbb R}
\def\N{\mathbb N}
\def\Rn{\mathbb R^n}
\def\Z{\mathbb Z}
\def\ve{\varepsilon}

\title[A sharp version of the H\"ormander Multiplier Theorem]
{A sharp version of the H\"ormander Multiplier Theorem}

\thanks{  }

\author{Loukas Grafakos}

\address{Department of Mathematics, University of Missouri, Columbia MO 65211, USA}
\email{grafakosl@missouri.edu}

\author{Lenka Slav\'ikov\'a}

\address{Department of Mathematics, University of Missouri, Columbia MO 65211, USA}
\email{slavikoval@missouri.edu}

\thanks{The first   author   acknowledges the  support of the Simons Foundation and of the University  of Missouri Research Board.}

\begin{abstract}
 We provide an improvement of the H\"ormander multiplier theorem in which the Sobolev space $L^r_s(\mathbb R^n)$ with 
   integrability  index $r$ and smoothness index $s>n/r$ is replaced by the Sobolev space with smoothness $s$ built upon the Lorentz space $L^{n/s,1}(\mathbb R^n)$.
   \end{abstract}

\maketitle


\bigskip


\section{Introduction}

Given a bounded function $\si$ on $\rn$, we define a   linear operator
$$
T_\si(f)(x) = \int_{\rn} \wh{f}(\xi) \si(\xi) e^{2\pi i x\cdot \xi}d\xi
$$
acting on Schwartz functions $f$  on $\rn$; here 
$\wh{f}(\xi) = \int_{\rn} f(x)   e^{-2\pi i x\cdot \xi}dx$ is the Fourier transform of $f$.  
An old problem  in harmonic analysis is to find optimal  sufficient  conditions on $\si$ to be 
 an $L^p$ Fourier multiplier, i.e.,  for   the 
operator $T_\si$ to admit  a bounded extension from $L^p(\rn)$ to itself for a given $p\in (1,\nf)$.

Mikhlin's~\cite{Mikhlin} classical multiplier theorem states that if the condition
\begin{equation}\label{10}
|\partial^\alpha \si(\xi)|\leq C_\alpha |\xi|^{-| \alpha|}, \qquad \xi\neq 0, 
\end{equation}
holds for all { multi-indices} $\al$ with size $|\al | \le [n/2]+1$, then $T_\si$ admits a bounded extension from
$L^p(\rn)$ to itself for all $1<p<\nf$. This theorem is well suited for dealing with multipliers whose derivatives have a singularity at one point, such as  functions which are homogeneous of degree zero and indefinitely differentiable on the unit sphere. 

An extension of the Mikhlin theorem was obtained by H\"ormander~\cite{Hoe}. It asserts the following: for $s>0$ let $(I-\De)^{s/2} $ denote the operator given on the Fourier transform by multiplication by
$(1+4\pi^2 |\xi|^2)^{s/2}$ and    let $\Psi$ be a Schwartz function whose Fourier transform is supported in the annulus  
$\{\xi: 1/2< |\xi|< 2\}$ and which satisfies $\sum_{j\in \mathbb Z} \wh{\Psi}(2^{-j}\xi)=1$ for all $\xi\neq 0$. If for some $1\le r\le 2$ and $s>n/r$,  $\si$ satisfies
\begin{equation}\label{2}
\sup_{k\in \mathbb Z} \big\|(I-\De)^{s/2} \big[ \wh{\Psi}\si (2^k \cdot)\big] \big\|_{L^r(\R^n) }<\infty,
\end{equation}
then $T_\si$ admits a bounded extension from
$L^p(\rn)$ to itself for all $1<p<\nf$. 

It is natural to ask whether condition~\eqref{2} can still guarantee that $\sigma$ is an $L^p$ Fourier multiplier for some $p\in (1,\infty)$ if $s\leq \frac{n}{2}$.  
Via an interpolation argument, 
Calder\'on and Torchinsky~\cite[Theorem 4.6]{CT} showed  that  $T_\si$ is bounded from $L^p(\bbr^n)$ to itself whenever condition~\eqref{2} holds with $p$ satisfying
$\big| \f 1p -\f 12 \big| <\f sn$ and $\big| \f 1p -\f 12 \big| = \f 1r $. It was observed in \cite{GraHeHonNg1} that the   assumption $\big| \f 1p -\f 12 \big| = \f 1r $ can be replaced by a weaker one, namely, by $\frac{1}{r}<\frac{s}{n}$.   Moreover, it is known that if  
$T_\si$ is bounded from $L^p(\bbr^n)$ to itself for every $\sigma$ satisfying~\eqref{2}, then $\big| \f 1p -\f 12 \big| \le \f sn$, see Hirschman~\cite{hirschman2}, Wainger~\cite{W}, Miyachi~\cite{Miy}, Miyachi and Tomita~\cite{MT}, Grafakos, He, Honz\'\i k, and Nguyen~\cite{GraHeHonNg1}.
In other words, when   $rs>n$, then the condition $\big| \f 1p -\f 12 \big| <\f sn$ is  essentially  optimal for assumption \eqref{2}. 
Observe also that the  condition $rs>n$ is dictated by the  embedding of $L^r_s(\R^n) \hookrightarrow L^\infty(\R^n)$.  
It is still unknown to us if $L^p$ boundedness holds  on the line $\big| \f 1p -\f 12 \big| =\f sn$ although other positive  results on this line for $1<p<2$ and  on $H^1$ can be found in Seeger~\cite{Seeger1}, ~\cite{Seeger2}.   

Unlike the Mikhlin multiplier theorem, the H\"ormander and Calder\'on-Torchinsky theorems can treat multipliers whose derivatives have infinitely many singularities, such as the multiplier
\begin{equation}\label{E:power_type}
\sigma(x)=\sum_{k\in \Z} \phi(2^{-k}x) |2^{-k}x-a_k|^\beta,
\end{equation}
where $\beta<0$, $\phi$ is a smooth function supported in the set $\{x\in \R^n: \frac{1}{2}<|x|<2\}$ and, for every $k\in \mathbb N$, $a_k\in \R^n$ belongs to
 the same set.  

In this paper, we improve the result of~\cite[Theorem 4.6]{CT} by replacing the Lebesgue space $L^r(\R^n)$, $r>\frac{n}{s}$, in condition~\eqref{2} by the  locally larger  Lorentz space $L^{\frac{n}{s},1}(\R^n)$, defined in terms of the norm
$$
\|f\|_{L^{\frac{n}{s},1}(\R^n)}=\int_0^\infty f^*(r)r^{\frac{s}{n}-1}\,dr. 
$$
Here, $f^*$ stands for the nonincreasing rearrangement of the function $f$, namely, for the unique nonincreasing left-continuous function on $(0,\infty)$ equimeasurable with $f$, given by the explicit expression
$$
f^*(t)=  \inf \big\{r\ge 0:\,\, |\{y\in \mathbb R^n:\,\, |f(y)|>r\}| < t \big\}\, .
$$
We point out that the Lorentz space $L^{\frac{n}{s},1}(\R^n)$ appears naturally in this context, since it is known to be, at least for integer values of $s$, locally the largest rearrangement-invariant function space such that membership of $(I-\Delta)^\frac{s}{2}f$   to this space forces 
  $f$   to be bounded, see~\cite{S, CP}.

\begin{theorem}\label{T:main_theorem} 
Let $\Psi$ be a Schwartz function on $\mathbb R^n$ whose Fourier transform is supported in the annulus $1/2<|\xi|<2$ 
and satisfies $\sum_{j\in \mathbb Z} \wh{\Psi} ( 2^{-j} \xi) =1$, $\xi\neq 0$.
Let $p\in (1,\infty)$, $n\in \N$, and let $s\in (0,n)$ satisfy
$$
\left|\frac{1}{p}-\frac{1}{2}\right|<\frac{s}{n}.
$$
Then for all functions $f$ in the Schwartz class of $\mathbb R^n$ we have the a priori estimate 
\begin{equation}\label{E:inequality}
\|T_\sigma f\|_{L^p(\Rn)} \leq C \sup_{j\in \Z} \big\|(I-\Delta)^{\frac{s}{2}}[\wh\Psi \si(2^j\cdot)\big]\big\|_{L^{\frac{n}{s},1}(\R^n)} \|f\|_{L^p(\Rn)}.
\end{equation}
\end{theorem}

As an application of Theorem~\ref{T:main_theorem} we show that the function $\sigma$ from~\eqref{E:power_type} continues to be an $L^p$ Fourier multiplier for any $p\in (1,\infty)$ if $|2^{-k}x-a_k|$ is replaced by $\log \frac{e4^n}{|2^{-k}x-a_k|^n}$. In fact, we can even allow an arbitrary iteration of logarithms in this example. 

\begin{example}\label{E:example}
Assume that $n\in \N$, $n\geq 2$, and $\beta<0$. 
Let $\phi$ be a smooth function supported in the set $A = \{x\in \R^n: 1/2<|x|<2\}$ and let $a_k\in A$, $k\in \mathbb Z$. Then the function
\begin{equation}\label{E:logarithmic_type}
\sigma(x)=\sum_{k\in \Z} \phi(2^{-k}x) \left(\log \frac{e4^n}{|2^{-k}x-a_k|^n}\right)^\beta
\end{equation}
is an $L^p$ Fourier multiplier for any $p\in (1,\infty)$.
\end{example}

To verify the statement of Example~\ref{E:example}, we fix a positive integer $s$ and observe that for any $j\in \Z$, 
\begin{align*}
&\|(I-\Delta)^{\frac{s}{2}}[\wh{\Psi}\sigma(2^j\cdot)]\|_{L^{\frac{n}{s},1}(\R^n)}
\leq \left\|(I-\Delta)^{\frac{s}{2}}\left[\wh{\Psi}(x) \phi(x) \left(\log \frac{e4^n}{|x-a_j|^n}\right)^\beta\right]\right\|_{L^{\frac{n}{s},1}(\R^n)}\\
&\qquad\qquad\qquad\qquad+\left\|(I-\Delta)^{\frac{s}{2}}\left[\wh{\Psi}(x) \phi(2x) \left(\log \frac{e4^n}{|2x-a_{j-1}|^n}\right)^\beta\right]\right\|_{L^{\frac{n}{s},1}(\R^n)}\\
&\qquad\qquad\qquad\qquad+\left\|(I-\Delta)^{\frac{s}{2}}\left[\wh{\Psi}(x) \phi\big(\frac{x}{2}\big) \left(\log \frac{e4^n}{|\frac{x}{2}-a_{j+1}|^n}\right)^\beta\right]\right\|_{L^{\frac{n}{s},1}(\R^n)} .
\end{align*}
In what follows, let us deal with the first term only, since the latter two terms can be estimated in a similar way. 

Fix $j\in \Z$ and denote
$$
f_j(x)=\wh{\Psi}(x) \phi(x) \left(\log \frac{e4^n}{|x-a_j|^n}\right)^\beta.
$$
Also, for any multiindex $\alpha$ satisfying $|\alpha|\geq 1$, let $\frac{\partial^\alpha}{\partial x^\alpha} f_j$ stand for the weak derivative of $f_j$ with respect to $\alpha$.
We have
$$
\left|\frac{\partial^\alpha}{\partial x^\alpha} f_j(x)\right| \leq C \chi_A(x) \left(\log \frac{e4^n}{|x-a_j|^{n}}\right)^{\beta-1} |x-a_j|^{-|\alpha|}.
$$
Since $|A|\leq 2^n \omega_n$, where $\omega_n$ stands for the volume of the unit ball in $\R^n$, the previous estimate implies
$$
\left(\frac{\partial^\alpha}{\partial x^\alpha} f_j\right)^*(t)\leq C \chi_{(0,2^n \omega_n)}(t) \left(\log \frac{e4^n \omega_n}{t}\right)^{\beta-1} t^{-\frac{|\alpha|}{n}},
$$
where the constant $C$ is independent of $j$. 
Therefore, if $s$ is a positive integer and $\alpha$ is a multiindex with $1\leq |\alpha|\leq s$, then
$$
\left(\frac{\partial^\alpha}{\partial x^\alpha} f_j\right)^*(t)\leq C \chi_{(0,2^n \omega_n)}(t) \left(\log \frac{e4^n \omega_n}{t}\right)^{\beta-1} t^{-\frac{s}{n}}.
$$
Consequently,
\begin{equation}\label{E:sobolev}
\sup_{1\leq |\alpha|\leq s} \left\|\frac{\partial^\alpha}{\partial x^\alpha} f_j\right\|_{L^{\frac{n}{s},1}(\R^n)} 
\leq C \int_0^{2^n \omega_n} \left(\log \frac{e4^n \omega_n}{t}\right)^{\beta-1} t^{-1}\,dt<\infty. 
\end{equation}
Since each $|f_j|$ is bounded by a constant independent of $j$ and compactly supported in the set $A$, we also have
$$
\|f_j\|_{L^{\frac{n}{s}}(\R^n)}\leq C<\infty.
$$

It remains to observe that the quantity $\|(I-\Delta)^{\frac{s}{2}}f_j\|_{L^{\frac{n}{s},1}(\R^n)}$ 
is equivalent to 
$$
\sum_{|\alpha|\leq s} \left\|\frac{\partial^\alpha}{\partial x^\alpha} f_j\right\|_{L^{\frac{n}{s},1}(\R^n)}.
$$ 
This can be proved in exactly the same way as the corresponding result for the Lebesgue spaces, see, e.g., \cite[Theorem 3, Chapter 5]{St1}. Therefore, we deduce that 
$$
\sup_{j\in \Z} \|(I-\Delta)^{\frac{s}{2}}[\wh{\Psi}\sigma(2^j\cdot)]\|_{L^{\frac{n}{s},1}(\R^n)}<\infty
$$ 
for any positive integer $s$. Theorem~\ref{T:main_theorem} now yields that $\sigma$ is an $L^p$ Fourier multiplier for any $p\in (1,\infty)$. 



Finally, notice that we can in fact replace the logarithm in~\eqref{E:logarithmic_type} by any iteration of logarithms, namely, we can consider the more general symbol
\begin{equation*}
\sigma(x)=\sum_{k\in \Z} \phi(2^{-k}x) \left(\underbrace{\log \cdots \log}_{\ell-\textup{times}} \frac{4^n\underbrace{e^{.{^{.{^{.{e}}}}}}}_{\ell- \textup{times}}}{|2^{-k}x-a_k|^n}\right)^\beta,
\end{equation*}
where $\ell$ is any positive integer. 
A computation similar to the one we performed above shows that $\sigma$ is an $L^p$ Fourier multiplier for any $p\in (1,\infty)$ as well.

\section{The main estimate}

In this section we show that   inequality~\eqref{E:inequality} holds for any $p\in (1,\infty)$ provided that $s\in (n/2,n)$, see Theorem~\ref{T:endpoint} below. This estimate will serve as one endpoint in the interpolation argument leading to the proof of Theorem~\ref{T:main_theorem}. The interpolation is the content of the next section. 

Let us start by recalling the definitions of two types of Lorentz spaces that will be used in the sequel. Suppose that $1<p<\infty$. Then, for any measurable function $f$ on $\R^n$, we define
$$
\|f\|_{L^{p,1}(\R^n)}=\int_0^\infty f^*(t) t^{\frac{1}{p}-1}\,dt
$$
and
$$
\|f\|_{L^{p,\infty}(\R^n)}=\sup_{t>0} f^*(t) t^{\frac{1}{p}}.
$$
It can be shown that
$$
\|f\|_{L^{p,1}(\R^n)}=p\int_0^\infty |\{x\in \R^n: |f(x)|>\lambda\}|^{\frac{1}{p}}\,d\lambda
$$
and
$$
\|f\|_{L^{p,\infty}(\R^n)}=\sup_{\lambda>0} \lambda |\{x\in \R^n: |f(x)|>\lambda\}|^{\frac{1}{p}}.
$$
The space $L^{p',\infty}(\R^n)$, where $p'=\frac{p}{p-1}$, is a kind of a measure theoretic dual of the space $L^{p,1}(\R^n)$, in the sense that the following form of H\"older's inequality
$$
\int_{\R^n} |fg| \leq \|f\|_{L^{p,1}(\R^n)} \|g\|_{L^{p',\infty}(\R^n)}
$$
holds. 

In what follows, $B(x,r)$   denotes the ball centered at point $x$ and having the radius $r$. If a ball of radius $r$ is centered at the origin, we shall denote it simply by $B_r$.
Let $q\geq 1$ be a real number. We   consider the centered maximal operator $M_{L^q}$ defined by
$$
M_{L^q} f(x) = \sup_{r>0} \left(\frac{1}{|B(x,r)|} \int_{B(x,r)} |f(y)|^q\,dy\right)^{\frac{1}{q}}.
$$
Observe that
$$
M_{L^q} f=(M |f|^q)^{\frac{1}{q}},
$$
where $M$ stands for the classical Hardy-Littlewood maximal operator. 

The crucial step towards proving Theorem~\ref{T:endpoint} is the following lemma, which can be understood as a sharp variant of ~\cite[Theorem 2.1.10]{CFA}.  

\begin{lemma}\label{L:lemma}
Assume that $n\in \N$, $s \in (0,n)$ and $q>\frac{n}{s}$. Then there is a positive constant $C$ depending on $n$, $s$ and $q$ such that for any $j\in \Z$ and any measurable function $f$ on $\Rn$,
\begin{equation}\label{E:estimate}
\left\|\frac{f(x+2^{-j}y)}{(1+|y|)^{s}}\right\|_{L^{\frac{n}{s},\infty}(\Rn)} \leq C M_{L^q} f(x), \quad x\in \Rn.
\end{equation}
\end{lemma}

\begin{proof}
We may assume, without loss of generality, that $j=0$ and $x=0$. Indeed, setting $g(y)=f(x+2^{-j}y)$, we obtain 
\begin{equation}\label{E:reduction1}
\left\|\frac{f(x+2^{-j}y)}{(1+|y|)^{s}}\right\|_{L^{\frac{n}{s},\infty}(\Rn)} = \left\|\frac{g(y)}{(1+|y|)^{s}}\right\|_{L^{\frac{n}{s},\infty}(\Rn)}
\end{equation}
and
\begin{align}\label{E:reduction2}
M_{L^q} f(x)
&=\sup_{r>0} \left(\frac{1}{|B(x,r)|} \int_{B(x,r)} |f(y)|^q\,dy\right)^{\frac{1}{q}}\\
\nonumber
&=\sup_{r>0} \left(\frac{1}{2^{jn}|B(x,r)|} \int_{B(0,2^j r)} |f(x+2^{-j}z)|^q\,dz\right)^{\frac{1}{q}}\\
\nonumber
&=\sup_{r'>0} \left(\frac{1}{|B(0,r')|} \int_{B(0,r')} |g(y)|^q\,dy\right)^{\frac{1}{q}}\\
\nonumber
&=M_{L^q} g(0). 
\end{align}
Hence, it suffices to show that for any measurable function $g$ on $\Rn$,
\begin{equation}\label{E:reduced_inequality}
\left\|\frac{g(y)}{(1+|y|)^{s}}\right\|_{L^{\frac{n}{s},\infty}(\Rn)} \leq C M_{L^q} g(0).
\end{equation}

If $M_{L^q} g(0)=\infty$, then inequality~\eqref{E:reduced_inequality} holds trivially, so we can assume in what follows that $M_{L^q} g(0)<\infty$. Since the case $M_{L^q} g(0)=0$ is trivial as well (as $g$ needs to vanish a.e.\ in this case), dividing the function $g$ by the positive constant $M_{L^q} g(0)$, we can in fact assume     that $M_{L^q} g(0)=1$.

Fix any $a>0$ and $k\in \N_0$. Then
\begin{align*}
|\{y\in B_{2^{k+1}}\setminus B_{2^{k}}: |g(y)|>a\}|
&\leq \frac{1}{a^q}\int_{B_{2^{k+1}}\setminus B_{2^{k}}} |g(y)|^q\,dy\\
&\leq \frac{|B_{2^{k+1}}| }{a^q}\cdot \frac{1}{|B_{2^{k+1}}|} \int_{B_{2^{k+1}}} |g(y)|^q\,dy
\leq \frac{\omega_n 2^{(k+1)n}}{a^q},
\end{align*}
where $\omega_n$ denotes the volume of the unit ball in $\Rn$. Combining this with the trivial estimate
$$
|\{y\in B_{2^{k+1}}\setminus B_{2^{k}}: |g(y)|>a\}| \leq \omega_n 2^{(k+1)n},
$$
we deduce that
\begin{align*}
&\left|\left\{y\in \Rn: \frac{|g(y)|}{(1+|y|)^s}>a\right\}\right|\\
&=\left|\left\{y\in B_1: \frac{|g(y)|}{(1+|y|)^s}>a\right\}\right| +\sum_{k=0}^\infty \left|\left\{y\in B_{2^{k+1}}\setminus B_{2^{k}}: \frac{|g(y)|}{(1+|y|)^s}>a\right\}\right|\\
&\leq \left|\left\{y\in B_1: |g(y)|>a\right\}\right| + \sum_{k=0}^\infty \left|\left\{y\in B_{2^{k+1}}\setminus B_{2^k} : |g(y)|> 2^{ks} a\right\}\right|\\
&\leq \left|\left\{y\in B_1: |g(y)|>a\right\}\right| +  \sum_{k=0}^\infty \omega_n 2^{(k+1)n} \min\left\{\frac{1}{2^{ksq} a^q}, 1 \right\}\\
&\leq \left|\left\{y\in B_1: |g(y)|>a\right\}\right| + \sum_{k\in \N_0: 2^k<\frac{1}{a^{ {1}/{s}}}} \omega_n 2^n \cdot 2^{kn} + \sum_{k\in \N_0: 2^k \geq \frac{1}{a^{ {1}/{s}}}} \frac{\omega_n 2^n}{a^q} \cdot 2^{k(n-sq)}\\
&\leq \left|\left\{y\in B_1: |g(y)|>a\right\}\right| +  \frac{C}{a^{\frac{n}{s}}}.
\end{align*}
Notice that in the last inequality we have used the fact that $n-sq<0$. Hence,
\begin{align*}
\left\|\frac{g(y)}{(1+|y|)^{s}}\right\|_{L^{\frac{n}{s},\infty}(\Rn)}
&=\sup_{a>0} a \left|\left\{y\in \Rn: \frac{|g(y)|}{(1+|y|)^s}>a\right\}\right|^{\frac{s}{n}}\\
&\leq \sup_{a>0} a \left|\left\{y\in B_1: |g(y)|>a\right\}\right|^{\frac{s}{n}}
+C\\
&=\|g\|_{L^{\frac{n}{s},\infty}(B_1)} +C\\
&\leq C' \|g\|_{L^q(B_1)}+C\\
&\leq C' \omega_n^{\frac{1}{q}} M_{L^q}g(0) +C
\leq C^{''},
\end{align*}
where $C'>0$ is the constant from the embedding $L^q(B_1) \hookrightarrow L^{\frac{n}{s},\infty}(B_1)$. 
Since $M_{L^q}g(0)=1$, this proves~\eqref{E:reduced_inequality}, and in turn~\eqref{E:estimate} as well.
\end{proof}

\begin{theorem}\label{T:endpoint}
Let $p\in (1,\infty)$, $n\in \N$, $s\in (\frac{n}{2},n)$. Let $\Psi$ be as in Theorem~\ref{T:main_theorem}. Then
\begin{equation}\label{E:hormander_lorentz}
\|T_\sigma f\|_{L^p(\Rn)} \leq C \sup_{j\in \Z} \big\|(I-\Delta)^{\frac{s}{2}}[\wh\Psi \si(2^j\cdot)\big]\big\|_{L^{\frac{n}{s},1}(\R^n)} \|f\|_{L^p(\Rn)}.
\end{equation}
\end{theorem}

\begin{proof}
Let 
$$
K=\sup_{j\in \mathbb Z} \big\|    (I-\De)^{\f{s }{2}}   \big[\wh{\Psi} \si(2^j\cdot)\big]\big\|_{L^{\frac{n}{s},1}(\R^n)} <\infty\, . 
$$
Introduce the function $\Theta$ satisfying 
$$
\wh{\Theta}(\xi)=\wh{\Psi}(\xi/2)+\wh{\Psi}(\xi)+\wh{\Psi}(2\xi),
$$
and observe that $\wh{\Theta}$ is equal to $1$ on the support of the function $\wh{\Psi}$. 

Let us denote by $\De_j$ and $\De_j^\Theta$ the Littlewood-Paley operators associated with $\Psi$ and $\Theta$, respectively. 
If $f$ is a Schwartz function on $\R^n$, then standard manipulations yield
\begin{align*}
\De_j T_\si (f)(x) &= \int_{\R^n} \wh{f}(\xi) \wh{\Psi}(2^{-j} \xi) \si(\xi) e^{2\pi i x\cdot \xi}d\xi
= \int_{\R^n} (\De_j^{\Theta}f )\sphat{}\,(\xi) \wh{\Psi}(2^{-j} \xi)   \si(\xi) e^{2\pi i x\cdot \xi}d\xi\\
&=2^{jn}  \int_{\R^n} (\De_j^{\Theta}f )\sphat{}\,(2^{j} \xi') \wh{\Psi}( \xi')   \si(2^{j}\xi') e^{2\pi i  x\cdot 2^{j}\xi'} d\xi' \\
&=\int_{\R^n} (\De_j^{\Theta}f )( x+2^{-j}y )  \big[\wh{\Psi}   \si(2^{j}\cdot)\big]\sphat\,  (y) \, dy\\
&=\int_{\R^n} \f{  (\De_j^{\Theta}f )( x+2^{-j}y )}{(1+|y|)^s}  (1+|y|)^s\big[\wh{\Psi}   \si(2^{j}\cdot)\big]\sphat\,  ( y) \, dy.
\end{align*}
By the H\"older inequality in Lorentz spaces, we therefore obtain
$$
|\De_j T_\si (f)(x)|\leq \left\|\f{(\De_j^{\Theta}f )( x+2^{-j}y )}{(1+|y|)^s}\right\|_{L^{\frac{n}{s},\infty}(\Rn)} \left\|(1+|y|)^s\big[\wh{\Psi}   \si(2^{j}\cdot)\big]\sphat\,  ( y)\right\|_{L^{(\frac{n}{s})',1}(\Rn)}.
$$

Since $\frac{n}{s}<2$, we can find a real number $q$ such that $\frac{n}{s}<q<2$. Lemma~\ref{L:lemma} now yields that
$$
\left\|\f{(\De_j^{\Theta}f )( x+2^{-j}y )}{(1+|y|)^s}\right\|_{L^{\frac{n}{s},\infty}(\Rn)} \leq C M_{L^q}(\De_j^{\Theta} f)(x).
$$

Using boundedness properties of the Fourier transform, we deduce that
\begin{align*}
\left\|(1+|y|)^s\big[\wh{\Psi}   \si(2^{j}\cdot)\big]\sphat\,  ( y)\right\|_{L^{(\frac{n}{s})',1}(\Rn)}
&\leq C \left\|(1+|y|^2)^{\frac{s}{2}} \big[\wh{\Psi}   \si(2^{j}\cdot)\big]\sphat\,  ( y)\right\|_{L^{(\frac{n}{s})',1}(\Rn)}\\
&\leq C \big\|    (I-\De)^{\f{s }{2}}   \big[\wh{\Psi} \si(2^j\cdot)\big]\big\|_{L^{\frac{n}{s},1}(\R^n)}
\leq CK.
\end{align*}
Altogether, we obtain the estimate
$$
|\De_j T_\si (f)|(x) \leq CK M_{L^q}(\De_j^{\Theta} f)(x).
$$

Assume that $p\geq 2$.
Then we get, by applying the Littlewood-Paley theorem and the Fefferman-Stein inequality (notice that $\frac{p}{q}\geq \frac{2}{q}>1$),
\begin{align*}
\big\| T_\si(f) \big\|_{L^p(\R^n)} 
&\le C \Big\| \Big( \sum_{j \in \mathbb Z} |\De_j T_\si (f) |^2 \Big)^{\f12} \Big\|_{L^p(\Rn)}
\leq C K \Big\| \Big( \sum_{j \in \mathbb Z} |M_{L^q}(\De_j^{\Theta} f) |^2 \Big)^{\f12} \Big\|_{L^p(\Rn)}\\
&=C K \Big\|\Big( \sum_{j \in \mathbb Z} (M (|\De_j^{\Theta} f|^q)^{\frac{2}{q}}\Big)^{\frac{q}{2}} \Big\|_{L^{\frac{p}{q}}(\Rn)}^{\frac{1}{q}}
\leq C K \Big\|\Big( \sum_{j \in \mathbb Z} |\De_j^{\Theta} f|^{q \cdot \frac{2}{q}}\Big)^{\frac{q}{2}} \Big\|_{L^{\frac{p}{q}}(\Rn)}^{\frac{1}{q}}\\
&= C K \Big\|\Big( \sum_{j \in \mathbb Z} |\De_j^{\Theta} f|^{2}\Big)^{\frac{1}{2}} \Big\|_{L^{p}(\Rn)}
\leq C K \|f\|_{L^p(\Rn)}.
\end{align*}

If $p\in (1,2)$ then the result follows by duality.
\end{proof}

\section{Interpolation}

Our main goal in this section will be to prove the following theorem.

\begin{theorem}\label{T:interpolation}
Suppose that $1<p_1<\infty$ and $0<s_1 <n$. If 
\begin{equation}\label{E:assumption}
\|T_\si f\|_{L^{p_1}(\R^n)} \leq C \sup_{j\in \Z} \|(I-\De)^{\frac{s_1}{2}}[\wh{\Psi} \si(2^j\cdot)]\|_{L^{\frac{n}{s_1},1}(\R^n)} \|f\|_{L^{p_1}(\R^n)},
\end{equation}
then
$$
\|T_\si f\|_{L^{p}(\R^n)} \leq C \sup_{j\in \Z} \|(I-\De)^{\frac{s}{2}}[\wh{\Psi} \si(2^j\cdot)]\|_{L^{\frac{n}{s},1}(\R^n)} \|f\|_{L^p(\R^n)}
$$
for any $1<p<\infty$ and $0<s<s_1$ satisfying
\begin{equation}\label{E:assumption_ps}
\frac{1}{s}\left|\frac{1}{p} - \frac{1}{2}\right| <\frac{1}{s_1} \left|\frac{1}{p_1}-\frac{1}{2}\right|.
\end{equation}
\end{theorem}

Assuming Theorem~\ref{T:interpolation}, and using the estimate from Theorem~\ref{T:endpoint} as the assumption~\eqref{E:assumption}, we finish the proof of our main result, Theorem~\ref{T:main_theorem}, as follows.

\begin{proof}[Proof of Theorem~\ref{T:main_theorem}]
If $s\in (\frac{n}{2},n)$, then inequality~\eqref{E:inequality} follows from Theorem~\ref{T:endpoint}. If $s\leq \frac{n}{2}$, then we denote
$$
\alpha=\frac{1}{s}\left|\frac{1}{p}-\frac{1}{2}\right|. 
$$
Since $\alpha \in (0,\frac{1}{n})$, we can find $p_1\in (1,\infty)$ and $s_1\in (\frac{n}{2},n)$ such that 
$$
\alpha <\frac{1}{s_1}\left|\frac{1}{p_1}-\frac{1}{2}\right|. 
$$
A combination of Theorems~\ref{T:endpoint} and~\ref{T:interpolation} thus yields the desired assertion~\eqref{E:inequality}.
\end{proof}

Let us now focus on the proof of Theorem~\ref{T:interpolation}. The main idea of the proof consists in applying a complex interpolation between the estimate~\eqref{E:assumption} and the usual $L^2$ estimate implied by the Plancherel theorem. 

To prove Theorem~\ref{T:interpolation} we shall need a few auxiliary results. With start by recalling the classical three lines lemma.

\begin{lemma}[{\cite{CFA, hirschman}}]\label{L:ThreeLines} 
  Let $F$ be analytic on the open strip $S=\{z\in\mathbb{C}\ :\ 0<\Re(z)<1\}$ and continuous on its closure.
  Assume that for every  $0\le \tau \le 1$ there exists a function $A_\tau$ on the { real} line such that
$$
    | F(\tau+it) |  \le A_\tau(t)    \qquad \textup{  for all $t\in\mathbb{R}$,}
$$
and suppose that there exist constants   $A>0$ and $0<a<\pi$ such that for all $t\in \mathbb R$ we have
  $$
0<  A_\tau(t) \le \exp \big\{ A e^{a |t|} \big\} \, .
  $$
 Then   for  $0<\theta<1 $   we have
  $$
  \abs{F(\theta )}\le \exp\left\{
  \dfrac{\sin(\pi \theta)}{2}\int_{-\infty}^{\infty}\left[
  \dfrac{\log |A_0(t ) | }{\cosh(\pi t)-\cos(\pi\theta)}
  +
  \dfrac{\log | A_1(t )| }{\cosh(\pi t)+\cos(\pi\theta)}
  \right]dt
  \right\}\, .
  $$
\end{lemma}
We point out that in calculations it is crucial to note that
\begin{equation}\label{ide}
\dfrac{\sin(\pi \theta)}{2}\int_{-\infty}^{\infty}
  \dfrac{dt }{\cosh(\pi t)-\cos(\pi\theta)}  =1-\theta\, , \quad
  \dfrac{\sin(\pi \theta)}{2}\int_{-\infty}^{\infty}
  \dfrac{dt }{\cosh(\pi t)+\cos(\pi\theta)} = \theta.
\end{equation}

\medskip

We shall also need the following lemma.

\begin{lemma}\label{L:016}
  Let $1<p, p_1<\infty$ be related as in $1/p=(1-\theta)/2+\theta/p_1$ for some $\theta \in (0,1)$. 
  Given $f\in {\mathscr C}_0^\nf(\mathbb R^n)$ and   $\ve>0,$ there exist  
  smooth functions $h_j^\ve$, $j=1,\dots, N_\ve$, supported in  cubes   on $\mathbb R^n$ with pairwise disjoint interiors,   
  and nonzero complex constants $c_j^\ve$ such that  the functions
  \begin{equation}\label{E:form}
  f_z^\ve  =  \sum_{j=1}^{N_\ve} |c_j^\ve|^{\frac p{2} (1-z) +   \frac p{p_1}  z}  \, h_j^\ve
  \end{equation}
satisfy
$$
 \big\|{f_\theta^\ve-f}\big\|_{L^2(\R^n)}<  \ve  
 $$
and 
$$
  \|{f_{it}^\ve}\|_{L^{2}(\R^n)} \leq   \left(\|f \|_{L^p(\R^n)} +\ve\right)^{\frac{p}{2}}  \, , \quad
  \|{f_{1+it}^\ve}\|_{L^{p_1}(\R^n)}\leq     \left(\|f \|_{L^p(\R^n)}  +\ve\right)^{\frac{p}{p_1}}\,.
$$
\end{lemma}

\begin{proof}
Given $f\in {\mathscr C}_0^\nf(\mathbb R^n)$  and   $ \ve>0$, by   uniform continuity
there are $N_\ve$   cubes $Q_j^\ve$ (with disjoint interiors) and constants $c_j^\ve $ such that
$$
\Big\| f - \sum_{j=1}^{N_\ve} c_j^\ve \chi_{Q_j^\ve} \Big\|_{L^2(\R^n)} + \Big\| f - \sum_{j=1}^{N_\ve} c_j^\ve \chi_{Q_j^\ve} \Big\|_{L^p(\R^n)}  <\ve    \, .  
$$
Find nonnegative smooth functions   $g_j^\ve\le \chi_{Q_j^\ve}$ such that
$$
\Big\| \sum_{j=1}^{N_\ve} c_j^\ve (  g_j^\ve-    \chi_{Q_j^\ve}   )  \Big\|_{L^2(\R^n)} + \Big\| \sum_{j=1}^{N_\ve} c_j^\ve (  g_j^\ve-    \chi_{Q_j^\ve}   )  \Big\|_{L^p(\R^n)}  <\ve.
$$
Let $\phi_j^\ve$ be the argument of the complex number $c_j^\ve$. 
Set $h_j^\ve = e^{i\phi_j^\ve} g_j^\ve$ and
notice that 
$ f_\theta^\ve = \sum_{j=1}^{N_\ve} |c_j^\ve| h_j^\ve =  \sum_{j=1}^{N_\ve}  c_j^\ve  g_j^\ve $  satisfies 
$$
 \big\|{f_\theta^\ve-f}\big\|_{L^2(\R^n)} + \big\|{f_\theta^\ve-f}\big\|_{L^p(\R^n)}<  \ve .
 $$
We also observe that
\begin{align*}
\|f_{it}^\ve\|_{L^2(\R^n)}^2
&=\|f_{1+it}^\ve\|_{L^{p_1}(\R^n)}^{p_1}=\|f_\theta^\ve\|_{L^p(\R^n)}^p\\
&\leq \left(\|f_\theta^\ve - f\|_{L^p(\R^n)} +\|f\|_{L^p(\R^n)}\right)^p \leq 
\left(\|f\|_{L^p(\R^n)}+\ve\right)^p,
\end{align*}
as claimed.
\end{proof}

The next three lemmas generalize results which are well known in the context of Lebesgue spaces into the setting of Lorentz spaces $L^{p,1}(\R^n)$.

\begin{lemma}\label{L:sobolev_embedding}
Let $0<s<n$. Then
$$
\|(I-\De)^{-\frac{s}{2}} f\|_{L^\infty(\R^n)} \leq C(n) \frac{s}{n-s} \|f\|_{L^{\frac{n}{s},1}(\R^n)}. 
$$
\end{lemma}

\begin{proof}
Let $G_s$ be the function defined for any $x\in \R^n$ by
$$
G_s(x)=\frac{1}{(4\pi)^{\frac{s}{2}}\Gamma(\frac{s}{2})} \int_0^\infty e^{-\frac{\pi|x|^2}{\delta}} e^{-\frac{\delta}{4\pi}} \delta^{\frac{-n+s}{2}} \frac{\,d\delta}{\delta}.
$$
It is not difficult to show that $G_s(x)\leq C(n)\frac{s}{n-s} |x|^{-n+s}$. Therefore,
\begin{align*}
|(I-\De)^{-\frac{s}{2}} f(x)|=|G_s\ast f(x)|
&\leq \int_{\R^n} G_s(y) |f(x-y)|\,dy
\leq \|G_s\|_{L^{(\frac{n}{s})',\infty}(\R^n)} \|f\|_{L^{\frac{n}{s},1}(\R^n)}\\
&\leq C(n) \frac{s}{n-s} \|f\|_{L^{\frac{n}{s},1}(\R^n)}.
\end{align*}
\end{proof}

\begin{lemma}\label{L:interpolation}
Let $1<a<b<\infty$. Then, for any $p\in (a,b)$ and any $t\in \R$, 
$$
\|(I-\De)^{-it}f\|_{L^{p,1}(\R^n)} \leq C(n,a,b) (1+|t|)^{\frac{n}{2}+1} \|f\|_{L^{p,1}(\R^n)}. 
$$
\end{lemma}

\begin{proof}
Set $b_0=2b$. By the H\"ormander multiplier theorem, one has
$$
\|(I-\De)^{-it}f\|_{L^{1,\infty}(\R^n)} \leq C(n) (1+|t|)^{\frac{n}{2}+1} \|f\|_{L^{1}(\R^n)} 
$$
and
$$
\|(I-\De)^{-it}f\|_{L^{b_0}(\R^n)} \leq C(n,b) (1+|t|)^{\frac{n}{2}+1} \|f\|_{L^{b_0}(\R^n)}.
$$
Notice that the second estimate implies, in particular, the corresponding weak-type inequality. An interpolation between these two estimates using the Marcinkiewicz interpolation theorem~\cite[Chapter 4, Theorem 4.13]{BS} yields the required assertion.
\end{proof}

\begin{lemma}\label{L:interpolation2}
Let $1<p<\infty$ and $s>0$, and let $\Psi$ be as in Theorem~\ref{T:main_theorem}. Then we have the a priori estimate 
\begin{equation}\label{E:kato_ponce}
\|(I-\De)^{\frac{s}{2}}[\wh{\Psi}f]\|_{L^{p,1}(\R^n)} \leq C(n,s,p,\Psi) \|(I-\De)^{\frac{s}{2}}f\|_{L^{p,1}(\R^n)}. 
\end{equation}
\end{lemma}

\begin{proof}
Pick real numbers $p_0$, $p_1$ satisfying $1<p_0<p<p_1<\infty$. Denote by $T$ the linear operator defined by
$$
Tf=(I-\Delta)^{\frac{s}{2}} [\wh{\Psi}(I-\Delta)^{-\frac{s}{2}}f].
$$
Thanks to the Kato-Ponce inequality, $T$ is bounded on both $L^{p_0}(\R^n)$ and $L^{p_1}(\R^n)$, so, in particular, it is of weak type $(p_0,p_0)$ and $(p_1,p_1)$.  
By the Marcinkiewicz interpolation theorem~\cite[Chapter 4, Theorem 4.13]{BS}, $T$ is bounded on $L^{p,1}(\R^n)$, which yields~\eqref{E:kato_ponce}.
\end{proof}

The final auxiliary result we shall need is the following.

\begin{lemma}\label{L:fractional_maximal_function}
Let $0<a<s<n$. Then
\begin{equation}\label{E:sunrise}
\int_0^\infty (f^*(r) r^{\frac{s-a}{n}})^*(y) y^{\frac{a}{n}-1}\,dy \leq \frac{C(n)}{a} \int_0^\infty f^*(r) r^{\frac{s}{n}-1}\,dr. 
\end{equation}
\end{lemma}

\begin{proof}
Estimates of this type are known in the literature, see, e.g., \cite{EO}. For the convenience of the reader, we also provide an elementary proof of inequality~\eqref{E:sunrise}. The proof follows the ideas of~\cite[Section 9]{CPS}.

We may assume that 
$$
\int_0^\infty f^*(r) r^{\frac{s}{n}-1}\,dr <\infty.
$$
Then $f^*(r) r^{\frac{s}{n}}\leq C$, and thus $\lim_{r\to \infty} f^*(r) r^{\frac{s-a}{n}}=0$. Since the function $f^*$ is left-continuous, $\sup_{y\leq r<\infty} f^*(r) r^{\frac{s-a}{n}}$ is attained for any $y>0$ and the set
$$
M=\{y\in (0,\infty): \sup_{y\leq r<\infty} f^*(r) r^{\frac{s-a}{n}} > f^*(y) y^{\frac{s-a}{n}}\}
$$
is open. Hence, $M$ is a countable union of open intervals, namely, $M=\bigcup_{k\in S} (a_k,b_k)$, where $S$ is a countable set of positive integers. Also, observe that if $y\in (a_k,b_k)$, then $\sup_{y\leq r<\infty} f^*(r) r^{\frac{s-a}{n}}=f^*(b_k) b_k^{\frac{s-a}{n}}$. We have
\begin{align*}
\int_0^\infty (f^*(r) r^{\frac{s-a}{n}})^*(y) y^{\frac{a}{n}-1}\,dy
&\leq \int_0^\infty \sup_{y\leq r<\infty} f^*(r) r^{\frac{s-a}{n}} y^{\frac{a}{n}-1}\,dy\\
&=\int_{(0,\infty)\setminus \cup_{k\in S} (a_k,b_k)} f^*(y) y^{\frac{s}{n}-1}\,dy
+\sum_{k\in S} f^*(b_k) b_k^{\frac{s-a}{n}} \int_{a_k}^{b_k} y^{\frac{a}{n}-1}\,dy.
\end{align*}

Furthermore, for every $k\in S$,
\begin{align*}
f^*(b_k) b_k^{\frac{s-a}{n}} \int_{a_k}^{b_k} y^{\frac{a}{n}-1}\,dy
&\leq f^*(b_k) b_k^{\frac{s-a}{n}}\int_{\max(a_k,\frac{b_k}{2})}^{b_k} y^{\frac{a}{n}-1}\,dy \cdot \frac{\int_0^{b_k} y^{\frac{a}{n}-1}\,dy}{\int_{\frac{b_k}{2}}^{b_k} y^{\frac{a}{n}-1}\,dy}\\
&=\frac{1}{1-(\frac{1}{2})^{\frac{a}{n}}} f^*(b_k) b_k^{\frac{s-a}{n}} \int_{\max(a_k,\frac{b_k}{2})}^{b_k} y^{\frac{a}{n}-1}\,dy\\
&\leq \frac{2^{\frac{s-a}{n}}}{1-(\frac{1}{2})^{\frac{a}{n}}} \int_{a_k}^{b_k} f^*(y) y^{\frac{s}{n}-1}\,dy\\
&\leq \frac{C(n)}{a} \int_{a_k}^{b_k} f^*(y) y^{\frac{s}{n}-1}\,dy.
\end{align*}
Therefore,
\begin{align*}
\int_0^\infty (f^*(r) r^{\frac{s-a}{n}})^*(y) y^{\frac{a}{n}-1}\,dy
&\leq \int_0^\infty f^*(y) y^{\frac{s}{n}-1}\,dy +\frac{C(n)}{a} \sum_{k\in S} \int_{a_k}^{b_k} f^*(y) y^{\frac{s}{n}-1}\,dy\\
&\leq \frac{C(n)}{a} \int_0^\infty f^*(y) y^{\frac{s}{n}-1}\,dy.
\end{align*}
\end{proof}

We are now in a position to prove Theorem~\ref{T:interpolation}. We shall need the notion of a measure preserving transformation. We say that a mapping $h: \R^n \rightarrow (0,\infty)$ is measure preserving if, whenever $E$ is a measurable subset of $(0,\infty)$, the set $h^{-1}E=\{x\in \R^n: h(x)\in E\}$ is a measurable subset of $\R^n$ and the $n$-dimensional Lebesgue measure of $h^{-1}E$ is equal to the one-dimensional Lebesgue measure of $E$. For more details on measure preserving transformations, see, e.g., \cite[Chapter 2, Section 7]{BS}.

\begin{proof}[Proof of Theorem~\ref{T:interpolation}]
We first observe that, by~\eqref{E:assumption_ps}, we have $p_1\neq 2$. In fact, we can assume that $1<p_1<2$ and $1<p\leq 2$, otherwise the result will follow by duality. Further, if $p=2$ then Theorem~\ref{T:interpolation} is a consequence of Plancherel's theorem and of the Sobolev embedding from Lemma~\ref{L:sobolev_embedding}, so it is sufficient to focus on the case $p<2$ in what follows. Define
$$
\theta=\frac{\frac{1}{p}-\frac{1}{2}}{\frac{1}{p_1}-\frac{1}{2}}. 
$$
The assumption~\eqref{E:assumption_ps} yields $\theta \in (0,\frac{s}{s_1})$, and therefore
$$
\theta = \frac{s-s_0}{s_1-s_0}
$$
for some $s_0\in (0,s)$. Fix a function $\si$ satisfying 
\begin{equation}\label{E:assumption_si}
\sup_{j\in \Z} \|(I-\De)^{\frac{s}{2}}[\wh{\Psi} \si(2^j\cdot)]\|_{L^{\frac{n}{s},1}(\R^n)}<\infty,
\end{equation}
and denote $\varphi_j=(I-\De)^{\frac{s}{2}}[\wh{\Psi} \si(2^j\cdot)]$, $j\in \Z$. Thanks to~\eqref{E:assumption_si}, we have $\lim_{r\to \infty} \varphi_j^*(r)=0$. By~\cite[Chapter 2, Corollary 7.6]{BS}, there is a measure preserving transformation $h_j: \R^n \rightarrow (0,\infty)$ such that $|\varphi_j|=\varphi_j^* \circ h_j$. 

For a complex number $z$ with $0\leq \Re(z)\leq 1$, we define
\begin{equation}\label{E:definition}
\si_z(\xi)
=\sum_{j\in \Z} (I - \De)^{-\frac{s_0(1-z)+s_1z}{2}} [\varphi_j h_j^{\frac{s-(1-z)s_0-zs_1}{n}}](2^{-j}\xi) \wh{\Phi}(2^{-j}\xi),
\end{equation}
where $\wh{\Phi}$ is a Schwartz function supported in the set $\{\xi \in \R^n: \frac{1}{4} \leq |\xi| \leq 4\}$ and $\wh{\Phi}\equiv 1$ on the support of $\wh{\Psi}$. 

Fix $f, g\in  \mathscr C_0^\nf$. Given $\ve>0$, let $f_z^\ve$ and $g_z^\ve$ be functions having the form~\eqref{E:form}, with $f$ replaced by $g$ and with $p$ replaced by $p'$ in the latter case, satisfying
  $  \norm{f_\theta^\ve-f}_{L^2(\R^n)}<\ve$, $ \norm{g_\theta^\ve-g}_{L^{2}(\R^n)}<\ve, $ 
  and
  \begin{align}\label{E:fg}
 & \norm{f_{it}^\ve}_{L^{2}(\R^n)}\le   \big( \norm{f}_{L^{p}(\R^n)}+\ve\big)^{\frac {p}{2}} ,\quad
  \norm{f_{1+it}^\ve}_{L^{p_1}(\R^n)}\le   \big( \norm{f}_{L^{p}(\R^n)}+\ve\big)^{\frac {p}{p_1}},\\
	\nonumber
 & \norm{g_{it}^\ve}_{L^{2}(\R^n)}\le   \big( \norm{g}_{L^{p'}(\R^n)}+\ve\big)^{\frac {p'}{2}},\quad
  \norm{g_{1+it}^\ve}_{L^{p_1'}(\R^n)}\le  \big( \norm{g}_{L^{p'}(\R^n)}+\ve\big)^{\frac {p'}{p_1'}}.
  \end{align}
Recall that the existence of these functions is guaranteed by Lemma \ref{L:016}.
For a complex number $z$ with $0\leq \Re(z) \leq 1$, define
  \begin{align*}
  F(z) =& \int_{\R^n} T_{\sigma_z}(f_z^\ve) {g}_z^\ve\; dx
	=\int_{\R^n} \si_z(\xi) \wh{f^{\ve}_{z}}(\xi) \wh{g^{\ve}_{z}}(\xi)\, d\xi.
  \end{align*}
It is straightforward (but rather tedious) to verify 
 that $F$ is analytic on the strip $S=\{z\in \mathcal C: 0 <\Re(z)<1\}$ and continuous on its closure.

Let us write $z=\tau +it$, $0\leq \tau \leq 1$ and $t\in \R$, and denote $s_\tau=s_0(1-\tau)+s_1\tau$. Then, applying Lemmas~\ref{L:sobolev_embedding} and \ref{L:interpolation} and using the fact that $h_j$ is measure preserving, we obtain
\begin{align*}
\|\si_z\|_{L^\infty(\R^n)}
&\leq C(n) \sup_{j\in \Z} \|(I-\De)^{-\frac{s_0(1-z)+s_1z}{2}}[\varphi_j h_j^{\frac{s-(1-z)s_0-zs_1}{n}}]\|_{L^\infty(\R^n)}\\
&\leq C(n) \frac{s_\tau}{n-s_\tau} \sup_{j\in \Z} \|(I-\De)^{-\frac{s_0(-it)+s_1it}{2}}[\varphi_j h_j^{\frac{s-(1-\tau-it)s_0-(\tau+it)s_1}{n}}]\|_{L^{\frac{n}{s_\tau},1}(\R^n)}\\
&\leq C(n,s_0,s_1) \frac{s_\tau}{n-s_\tau} (1+|t|)^{\frac{n}{2}+1} \sup_{j\in \Z} \|\varphi_j h_j^{\frac{s-(1-\tau-it)s_0-(\tau+it)s_1}{n}}\|_{L^{\frac{n}{s_\tau},1}(\R^n)}\\
&\leq C(n,s_0,s_1) (1+|t|)^{\frac{n}{2}+1} \sup_{j\in \Z} \|\varphi_j^*(r) r^{\frac{s-(1-\tau)s_0-\tau s_1}{n}}\|_{L^{\frac{n}{s_\tau},1}(0,\infty)}\\
&\leq C(n,s_0,s_1) (1+|t|)^{\frac{n}{2}+1} \sup_{j\in \Z} \|\varphi_j^*\|_{L^{\frac{n}{s},1}(0,\infty)}\\
&\leq C(n,s_0,s_1) (1+|t|)^{\frac{n}{2}+1} \sup_{j\in \Z} \|\varphi_j\|_{L^{\frac{n}{s},1}(\R^n)}.
\end{align*}
Notice that if $\tau \in [0,\theta)$, then the last but one inequality follows from Lemma~\ref{L:fractional_maximal_function}. 
Therefore,
\begin{align}\label{E:l_infty_estimate}
|F(z)|&\leq \|\si_z\|_{L^\infty(\R^n)} \|f^\ve_z\|_{L^2(\R^n)} \|g^\ve_z\|_{L^2(\R^n)}\\
\nonumber
&\leq C(n,s_0,s_1) (1+|t|)^{\frac{n}{2}+1} \sup_{j\in \Z} \|\varphi_j\|_{L^{\frac{n}{s},1}(\R^n)} \|f^\ve_z\|_{L^2(\R^n)} \|g^\ve_z\|_{L^2(\R^n)}.
\end{align}
Since $\|f^\ve_z\|_{L^2(\R^n)} \|g^\ve_z\|_{L^2(\R^n)}$ can be bounded from above by a constant independent of $z$, the previous estimate yields
\begin{equation}\label{E:assumption_interpolation}
|F(z)|\leq C(n,s_0,s_1,p,p_1,\ve,f,g) (1+|t|)^{\frac{n}{2}+1} \sup_{j\in \Z} \|\varphi_j\|_{L^{\frac{n}{s},1}(\R^n)} \leq \exp\{A e^{a|t|}\}
\end{equation}
for a suitable choice of constants $A>0$ and $a\in (0,\pi)$.
Also, if $z=it$, $t\in \R$, then~\eqref{E:l_infty_estimate} combined with~\eqref{E:fg} yield
\begin{equation}\label{E:endpoint0}
|F(it)|\leq C(n,s_0,s_1) (1+|t|)^{\frac{n}{2}+1} \big( \norm{f}_{L^{p}(\R^n)}+\ve\big)^{\frac p{2}} \big( \norm{g}_{L^{p'}(\R^n)}+\ve\big)^{\frac{p'}{2}} \sup_{j\in \Z} \|\varphi_j\|_{L^{\frac{n}{s},1}(\R^n)}.
\end{equation}
Finally, by the H\"older inequality and by~\eqref{E:assumption},
\begin{align*}
|F(1+it)|&\leq \|T_{\si_{1+it}}(f^\ve_{1+it})\|_{L^{p_1}(\R^n)} \|g^\ve_{1+it}\|_{L^{p'_1}(\R^n)}\\
&\leq C\sup_{j\in \Z} \|(I-\De)^{\frac{s_1}{2}} [\wh{\Psi} \si_{1+it}(2^j\cdot )]\|_{L^{\frac{n}{s_1},1}(\R^n)} \|f^\ve_{1+it}\|_{L^{p_1}(\R^n)} \|g^\ve_{1+it}\|_{L^{p'_1}(\R^n)}.
\end{align*}
Notice that $\wh{\Psi} \si_{1+it}(2^k\cdot )$ picks up only those terms $j$ of~\eqref{E:definition} which differ from $k$ by at most two units. For simplicity, we may therefore take $j=k$ in the calculation below. We have
\begin{align*}
&\|(I-\De)^{\frac{s_1}{2}} [\wh{\Psi}(I - \De)^{-\frac{s_1+it(s_1-s_0)}{2}} [\varphi_j h_j^{\frac{s-s_1+it(s_0-s_1)}{n}}]]\|_{L^{\frac{n}{s_1},1}(\R^n)}\\
&\leq C \|(I-\De)^{\frac{s_1}{2}} [(I - \De)^{-\frac{s_1+it(s_1-s_0)}{2}} [\varphi_j h_j^{\frac{s-s_1+it(s_0-s_1)}{n}}]]\|_{L^{\frac{n}{s_1},1}(\R^n)}\\
&\leq C\|(I - \De)^{-\frac{it(s_1-s_0)}{2}} [\varphi_j h_j^{\frac{s-s_1+it(s_0-s_1)}{n}}]\|_{L^{\frac{n}{s_1},1}(\R^n)}\\
&\leq C (1+|t|)^{\frac{n}{2}+1}\|\varphi_j h_j^{\frac{s-s_1}{n}} \|_{L^{\frac{n}{s_1},1}(\R^n)}
=C (1+|t|)^{\frac{n}{2}+1} \|\varphi_j^*(r) r^{\frac{s-s_1}{n}}\|_{L^{\frac{n}{s_1},1}(0,\infty)}\\
&=C (1+|t|)^{\frac{n}{2}+1} \|\varphi_j^* \|_{L^{\frac{n}{s},1}(0,\infty)}
= C (1+|t|)^{\frac{n}{2}+1} \|\varphi_j \|_{L^{\frac{n}{s},1}(\R^n)}.
\end{align*}
Notice that in the previous estimate we consecutively used Lemmas~\ref{L:interpolation2} and \ref{L:interpolation} and the fact that $h_j$ is measure preserving. 
Therefore,
\begin{equation}\label{E:endpoint1}
|F(1+it)|\leq C (1+|t|)^{\frac{n}{2}+1} \sup_{j\in \Z} \|\varphi_j \|_{L^{\frac{n}{s},1}(\R^n)} (\|f\|_{L^{p}(\R^n)}+\ve)^{\frac{p}{p_1}} (\|g\|_{L^{p'}(\R^n)}+\ve)^{\frac{p'}{p_1'}}.
\end{equation}

A combination of~\eqref{E:assumption_interpolation}, \eqref{E:endpoint0}, \eqref{E:endpoint1} and Lemma~\ref{L:ThreeLines} yields
\begin{equation}\label{E:theta}
|F(\theta)|\leq C\sup_{j\in \Z} \|\varphi_j \|_{L^{\frac{n}{s},1}(\R^n)} (\|f\|_{L^{p}(\R^n)}+\ve) (\|g\|_{L^{p'}(\R^n)}+\ve).
\end{equation}

Observe that for every $\xi\neq 0$,
\begin{align*}
\sigma_\theta(\xi)
&=\sum_{j\in \Z} (I-\De)^{-\frac{s}{2}}[(I-\De)^{\frac{s}{2}}[\sigma(2^j\cdot)\wh{\Psi}]](2^{-j}\xi) \wh{\Phi}(2^{-j}\xi)\\
&=\sum_{j\in \Z} \sigma(\xi) \wh{\Psi}(2^{-j}\xi) \wh{\Phi}(2^{-j}\xi)
=\sum_{j\in \Z} \sigma(\xi) \wh{\Psi}(2^{-j}\xi)
=\sigma(\xi).
\end{align*}
Thus,
$$
  F(\theta) = \int_{\mathbb R^n} \sigma(\xi) \widehat{f_\theta^\ve}(\xi) \wh{ {g}_\theta^\ve}(\xi) \, d\xi \, .
  $$
Notice that
\begin{align*}
 &\bigg| \int_{\mathbb R^n} \sigma(\xi)  \widehat{f_\theta^\ve}(\xi) \wh{ {g}_\theta^\ve}(\xi) \, d\xi
- \int_{\mathbb R^n} \sigma(\xi) \widehat{f }(\xi) \wh{  g }(\xi) \, d\xi \bigg| \\
= &\bigg|  \int_{\mathbb R^n} \sigma(\xi) \Big[ \widehat{f_\theta^\ve}(\xi) \big(\wh{ {g}_\theta^\ve}(\xi)-\wh{  g }(\xi) \big)
 +  \widehat{g }(\xi) \big(\wh{ {f}_\theta^\ve}(\xi)-\wh{  f }(\xi) \big)  \Big]\, d\xi \bigg|  \\
 \le & \|\si\|_{\li(\R^n)}  \Big[ \|f_\theta^\ve\|_{L^2(\R^n)} \|{g}_\theta^\ve -g\|_{L^2(\R^n)} +
 \|g\|_{L^2(\R^n)} \|{f}_\theta^\ve -f\|_{L^2(\R^n)} \Big] \\
  \le & C\sup_{j\in \Z} \|(I-\De)^{\frac{s}{2}}[\wh{\Psi}\si(2^j\cdot)]\|_{L^{\frac{n}{s},1}(\R^n)}  \Big[ \|f^\ve_\theta \|_{L^2(\R^n)} \|{g}_\theta^\ve -g\|_{L^2(\R^n)} +
 \|g \|_{L^2(\R^n)} \|{f}_\theta^\ve -f\|_{L^2(\R^n)} \Big]\, .
\end{align*}
Recall that the functions $f_0^\ve$ and $g_0^\ve$ were chosen in such a way that ${f}_\theta^\ve -f$ and ${g}_\theta^\ve -g$ converge  to zero in $L^2(\R^n)$ as $\ve$ converges to $0$. Therefore, 
letting $\ve\to 0$ in~\eqref{E:theta} yields
$$
\bigg|\int_{\mathbb R^n} \sigma(\xi) \widehat{f }(\xi) \wh{  g }(\xi) \, d\xi \bigg|\le
C\, \sup_{j\in\mathbb Z}
  \big\|{(I-\Delta)^{\frac{s}2}[\sigma(2^j\cdot)\widehat{\Psi}]}\big\|_{L^{\frac{n}{s},1}(\R^n)}
\norm{f}_{L^{p}(\R^n)} \|g\|_{L^{p'}(\R^n)}.
$$
Taking the supremum over all functions $g\in L^{p'}(\R^n)$ with $\| g\|_{L^{p'}(\R^n)} \le 1$ we obtain
  $$
  \norm{T_{\sigma }(f)}_{L^p(\R^n)}\le C\, \sup_{j\in\mathbb Z}
  \big\|{(I-\Delta)^{\frac{s}2}[\sigma(2^j\cdot)\widehat{\Psi}]}\big\|_{L^{\frac{n}{s},1}(\R^n)}
\norm{f}_{L^{p}(\R^n)}.
  $$
The proof is complete.
\end{proof}

 \end{document}